\documentclass[11pt]{articlefederico}
\usepackage{graphicx}
\usepackage{amsfonts }
\usepackage{amsmath}
\usepackage{fullpage}
\usepackage{amssymb}
\usepackage{amsthm}
\usepackage{tikz}
\usepackage{lipsum}
\usepackage{mathrsfs}
\usepackage{hyperref}
\usepackage{float}
\usepackage{multicol}

\graphicspath{{figures/}}
\usepackage{array}
\usepackage{enumitem}
\usepackage[mathscr]{euscript}

\graphicspath{{figures/}}

\newcounter{counter}
\theoremstyle{definition}
\newtheorem{definition}[counter]{Definition}

\newtheorem{theorem}[counter]{Theorem}

\newtheorem{proposition}[counter]{Proposition}

\newcommand*{\e}{\mathsf{e}}
\newcommand*{\R}{\mathbb{R}}
\newcommand*{\Z}{\mathbb{Z}}
\newcommand*{\A}{\mathcal{A}}

\begin{document}

\title{\textsf{The Degree of a Tropical Root Surface of Type A }}

\author{
\textsf{Federico Ardila--Mantilla\footnote{\noindent \textsf{San Francisco State University, Universidad de Los Andes; federico@sfsu.edu. Supported by NSF Award DMS-2154279.}}}\\
\and \textsf{Mont Cordero Aguilar\footnote{\noindent \textsf{University of Washington, mont@uw.edu.}}}}

\date{}

\maketitle

%


\begin{abstract}
We prove that the tropical surface of the root system $A_{n-1}$ has degree 
$\frac{1}{2}n (n-1)(n-2)$.
\end{abstract}

\section{\textsf{Introduction}}

Tropical geometry was developed to answer questions in classical algebraic geometry combinatorially. Tropicalization converts a projective variety $V$ into a polyhedral complex $\mathrm{trop}(V)$ that, roughly speaking, records the behavior of $V$ at infinity. The tropical variety $\mathrm{trop}(V)$ retains a surprising amount of information about $V$, such as its dimension and  degree. 
Many important invariants of $\mathrm{trop}(V)$ can be computed using combinatorics and discrete geometry, thus giving computations of algebro-geometric invariants of $V$. 
For detailed introductions to tropical geometry, see \cite{BS14, maclagan2015introduction, mikhalkin2010tropical}.

Initially, tropical geometry was most interested in studying tropicalizations of algebraic varieties of importance. However, a more robust theory arises when one considers abstract tropical varieties, most of which do not arise via tropicalization. This is analogous to the situation in matroid theory, where a linear subspace $V$ of a vector space gives rise to a matroid $M_V$, but a  more robust theory arises when one considers all matroids, most of which do not arise from a linear subspace. (This is not just an analogy: matroids may be understood as the tropical fans of degree 1 \cite{AK06, fink}.) 

Tropical geometry is then a rich source of well motivated combinatorial problems of significance within and beyond combinatorics. A good theory needs good examples, and combinatorics is a rich source of tropical varieties. In this spirit, Ardila, Kato, McMillon, Perez, and Schindler \cite{poster} constructed tropical surfaces associated in a natural way to the classical root systems. They  proved that the tropical Laplacians of these surfaces have exactly one negative eigenvalue, as one might predict from the Hodge index theorem. 

The geometric protagonist of this paper is the tropical surface $S(A_{n-1})$ associated to the root system $A_{n-1}$ of the special linear Lie algebra $\mathfrak{sl}_n$. Our main result is that this surface has degree $\frac{1}{2}n (n-1)(n-2)$.

\section{\textsf{Background.}} 




Let $n$ be a positive integer and write $[ n ] = \{1,2,\ldots,n\}$. Let $\e_1, \ldots, \e_n$ be the standard basis of $\R^n$, and write $\e_S = \displaystyle \sum_{s \in S} \e_s$ for each subset $S \subseteq [n]$.

\subsection{\textsf{Root Systems}}
Let us begin by defining root systems and root polytopes.

\begin{definition}
   \cite{bourbaki} A \textbf{crystallographic root system} $\Phi$ is a set of vectors in $\mathbb{R}^n$ satisfying:
    	\begin{itemize}
    		\item For every root $\beta \in \Phi$ , the set $\Phi$ is closed under reflection across the hyperplane perpendicular to $\beta$.
    		\item For any two roots $\alpha,\beta \in \Phi, \text { the quantity } 2\frac{\langle \alpha, \beta \rangle}{\langle \alpha, \alpha \rangle}$ is an integer, where $\langle - , - \rangle$ is the standard inner product in $\mathbb{R}^n$.
    		\item If $\beta, c\beta \in \Phi$ for $c \in \R$, then $c=1 \text{ or } c=-1$.
    	\end{itemize}
    	\end{definition}
 
\begin{definition}
\cite{bourbaki}
    An \textbf{irreducible root system} is one that cannot be partitioned into the union of two proper subsets $\Delta = \Delta_1 \cup \Delta_2$, such that $\langle \alpha,\beta \rangle = 0$ for all $\alpha \in \Delta_1 \text{ and } \beta \in \Delta_2$.
\end{definition}

Root systems play a fundamental role in many areas of mathematics; for example, they are key to the classification of semisimple Lie algebras \cite{bourbaki}. 
 The irreducible root systems have been classified into four infinite \textbf{classical families} and five \textbf{exceptional} root systems. In this paper we focus on the most classical family:
 \[
A_{n-1} = \{ \e_i - \e_j : i,j \in [n], i \neq j\}.
\]
This is the root system of the special linear Lie algebra $\mathfrak{sl}_n$. 

 \begin{definition}
    The \textbf{root polytope} $P(\Phi)$ of a root system $\Phi$ is the convex hull of $\Phi$.
\end{definition}


%

\subsection{\textsf{Tropical Geometry}} \label{sec:tropgeom}

To define the root surfaces $S(A_{n-1})$ that interest us, we first introduce some basic definitions from tropical geometry.

A \textbf{cone} is a set of the form
\[
\textrm{cone}(v_1,\ldots, v_n) = \{\lambda_1 v_1 + \cdots + \lambda_n v_n : \lambda_1, \lambda_2, . . . , \lambda_n \geq 0\}
\]
for vectors $v_1, \ldots, v_n$ in $\R^d$. The cone is \textbf{rational} if it is generated by integer vectors. 
A (rational) \textbf{polyhedral fan} is a nonempty finite collection $\Sigma$ of (rational) cones in $\mathbb{R}^d$ 
such that every face of a cone in $\Sigma$ is also in $\Sigma$, and the intersection of any two cones in $\Sigma$ is a face of both of them.
A fan is  \textbf{pure} of dimension $d$ if all maximal faces are $d$-dimensional. We let $\Sigma^i$ denote the set of cones of $\Sigma$ of dimension $i$.
Tropical fans are those that meet the following balancing condition.
 
\begin{definition} \label{def:balcond}
\cite{maclagan2015introduction}
Let $\Sigma \subseteq \mathbb{R}^n$ be a rational polyhedral fan, pure of dimension $d$, with a choice of weight $w(\sigma) \in \mathbb{N}$ for each maximal cone $\sigma \in \Sigma^d$.

%

For each $(d-1)$-cone $\tau \in \Sigma^{d-1}$, consider the $(d-1)$-subspace $L_\tau \subseteq \R^n$ spanned by $\tau$, the induced $(d-1)$-lattice $L_{\tau,\mathbb{Z}} = L_\tau \cap \mathbb{Z}^n$, and the quotient $(n-d+1)$-lattice $N_\tau = \Z^n/L_{\tau,\mathbb{Z}}$. 
Each $d$-cone $\sigma \in \Sigma^d$ with $\sigma \supset \tau$ determines a ray $(\sigma+L_\tau)/L_\tau$ in $\R^n/L_\tau$. This ray is rational with respect to the lattice $N_\tau$; let $u_{\sigma/\tau}$ be the first lattice point on this ray. The fan $\Sigma$ is \textbf{balanced at $\tau$} if the following relation holds in $\R^n/L_\tau$:
    \[
 \sum_{\sigma \in \Sigma^d \, : \, \sigma \supset \tau} w(\sigma)u_{\sigma/\tau} = 0
 \]
    The fan $\Sigma$ is a  \textbf{tropical fan} if it is balanced at all faces of dimension $d - 1$.
\end{definition}

Tropical varieties are more general than tropical fans; see \cite{mikhalkin2010tropical} for a definition. \textbf{Tropical surfaces} are tropical varieties that are pure of dimension 2. In particular, 2-dimensional tropical fans are tropical surfaces.

\begin{definition}\label{def_root_surf}
   The \textbf{tropical root surface} $S(\Phi)$ of a root system $\Phi$ is the cone  over the 1-skeleton of $P(\Phi)$ with unit weights on all facets. It consists of:
   \begin{itemize}
       \item Rays: cone$(r)$ for each $r \in \Phi$.
       \item Facets: cone$(r,s)$ for each $r,s \in \Phi$ such that $rs$ is an edge of the root polytope $P(\Phi)$.
       \item Weights: $w(\sigma) = 1$ for every facet $\sigma$. 
   \end{itemize}
\end{definition}

 \begin{figure}[h!]
    \begin{center}
     \includegraphics[height = 4.5cm]{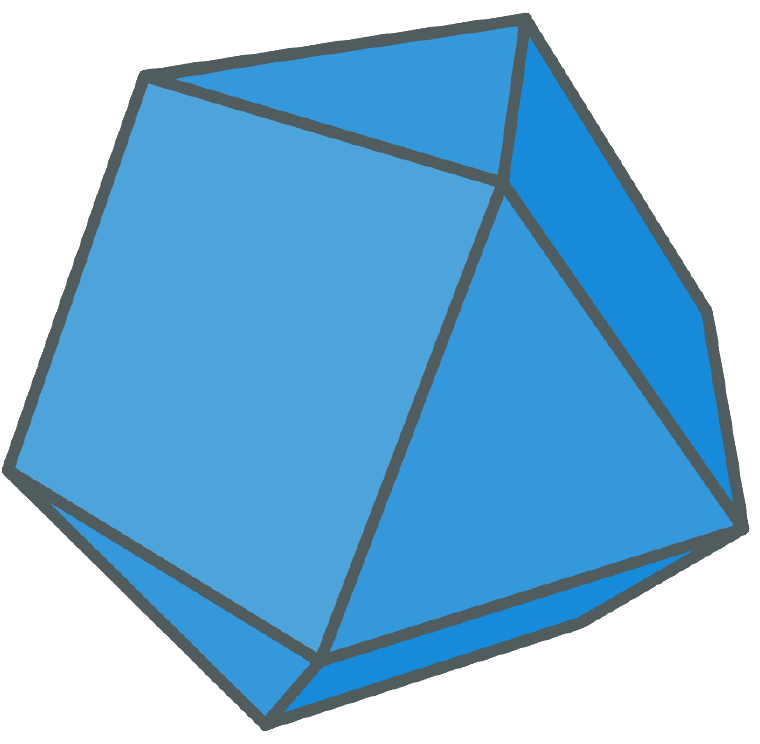} \qquad \qquad  
     \includegraphics[height = 4.5cm]{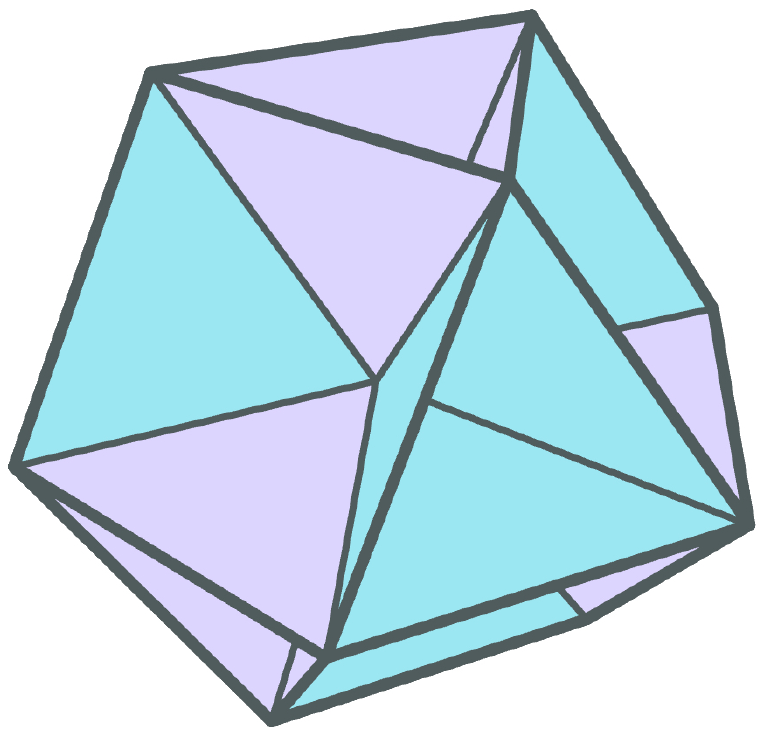} 
    \end{center}
    \caption{The root polytope $P(A_3)$ and the tropical root surface $S(A_3)$. \label{fig:A3} }
\end{figure}

Tropical root surfaces were introduced by \cite{poster, perez_2019, schindler_2017} by Federico Ardila, Chiemi Kato, Jewell McMillon, Maria Isabel Perez, and Anna Schindler. Figure \ref{fig:A3} shows the root polytope and the tropical surface of the root system $A_3$, with its cones truncated for visibility.

In classical algebraic geometry, the degree of an irreducible affine or projective variety of dimension $d$ is obtained by counting its intersection points with a generic linear space of codimension $d$. In tropical geometry, degree is defined similarly. The analog of a generic linear space is a generic shift of the standard tropical linear space of codimension $d$, 
which we now define.

\begin{definition}\label{std-trop-hyp}
    \cite{AK06, maclagan2015introduction} 
The \textbf{standard tropical linear space} $\Sigma_{n,n-d}$ is the tropical fan whose facets are the cones
\begin{eqnarray*}
\sigma_{i_1,i_2,\ldots,i_{n-d-1}}
&=& \{ x \in \mathbb{R}^n: x_{i_1} \geq x_{i_2} \geq \cdots  \geq x_{i_{n-d-1}}  \geq x_{i_{n-d}} = \cdots = x_{i_{n-1}} = x_{i_{n}} \} \\
&=& \textrm{cone}\{\e_{\{i_1\}}, \e_{\{i_1,i_2\}}, \ldots, \e_{\{i_1, \ldots, i_{n-d-1}\}}, \e_{[n]}\}
\end{eqnarray*}
for each choice of distinct $i_1,i_2,\ldots,i_{n-d-1} \in [n]$,
where every facet has weight 1.
\end{definition}

The fan $\Sigma_{n,n-d}$ described above is the \textbf{fine subdivision} of the \textbf{Bergman fan} of the uniform matroid $U_{n,n-d}$, as shown in \cite{AK06}. Its support (i.e., the union of all of its cones) is the set of vectors in $\R^n$ whose smallest $d+1$ entries are equal to each other.

\begin{definition}
Consider two tropical fans $\Sigma_1$ and $\Sigma_2$ in $\R^n$ of complementary dimensions $d_1$ and $d_2$; that is, $d_1+d_2 = n$. 

We say $\Sigma_1$ and $\Sigma_2$ \textbf{intersect transversally} if $\Sigma_1 \cap \Sigma_2$ is a finite union of points, and each such point $p$ can be written uniquely as $p = \sigma_1 \cap \sigma_2$ for facets $\sigma_1, \sigma_2$ of $\Sigma_1, \Sigma_2$, respectively. The \textbf{weight} of each intersection point $p$ is
\[
w(p) := w(\sigma_1)w(\sigma_2)  \, \, [\Z^n : L_{\sigma_1, \Z} + L_{\sigma_2, \Z}] 
\]
We call $\text{index}(p) := [\Z^n : L_{\sigma_1, \Z} + L_{\sigma_2, \Z}]$ the \textbf{index} of $p$. The \textbf{degree of the transversal intersection at $p$} is
\[
\Sigma_1 \cdot \Sigma_2 := \sum_{p \in \Sigma_1 \cap \Sigma_2} w(p).
\]

If $\Sigma_1$ and $\Sigma_2$ are balanced but do not necessarily intersect transversally, then $\Sigma_1$ and $\Sigma_2+v$ do intersect transversally 
for generic vectors $v \in \R^n$, and the balancing condition implies that the degree of their transversal intersection does not depend on $v$ \cite[Proposition 4.3.3, 4.3.6]{mikhalkin2010tropical}. Thus we define the \textbf{degree of the intersection} to be
\[
\Sigma_1 \cdot \Sigma_2 := (v+\Sigma_1) \cdot \Sigma_2
\]
for generic $v$.

Finally, the \textbf{degree of a tropical fan} $\Sigma$ in $\R^n$ of dimension $d$ is the degree of its intersection with the standard tropical linear space of codimension $d$:
\[
\deg \Sigma := \Sigma \cdot  \Sigma_{n,n-d}.
\]
\end{definition}
%

In practice, to find the degree of a tropical fan $\Sigma$, one chooses a convenient generic vector $v \in \R^n$ and performs the following steps.

\begin{enumerate}[label=\emph{\Alph*})]
    \item Find the intersections of $v+\Sigma$ with $\Sigma_{n,n-d}$.
    \item For each intersection point $p$ identify the cones $v + \sigma_1$ of $v + \Sigma$ and $\sigma_2$ of $\Sigma_{n,n-d}$ containing it, and find the weight of that intersection. 
        \item Find the degree of $\Sigma$ by adding the weights of the intersection points above.
\end{enumerate}

Step B) is easier when $\sigma_1$ and $\sigma_2$ are simplicial and saturated, with 
$\sigma_1 = \mathbb{R}_{\geq 0}\langle\alpha_1, \ldots, \alpha_{d_1} \rangle$,  $\sigma_2 = \mathbb{R}_{\geq 0}\langle\beta_1,\ldots,\beta_{d_2}\rangle$
and
$\sigma_1 \cap \Z^n = \mathbb{Z}_{\geq 0}\langle\alpha_1, \ldots, \alpha_{d_1} \rangle$, $\sigma_2 \cap \Z^n = \mathbb{Z}_{\geq }\langle\beta_1,\ldots,\beta_{d_2}\rangle$. In this case the index of the intersection $p$ can be computed as follows:
\[
\text{index}(p) =  |\det(\alpha_1, \ldots, \alpha_{d_1}, \beta_1, \ldots, \beta_{d_2})|.
\]

\section{\textsf{The Tropical Root Surface of Type A and its Degree}}

The following result first appeared in \cite{schindler_2017}, in a slightly different form. We include a proof for completeness.

\begin{proposition}
The tropical root surface $S(A_{n-1})$ is a tropical surface.
\end{proposition}

\begin{proof}
We verify the balancing condition 
for an arbitrary ray $r = \text{cone}(\e_i - \e_k)$. The maximal cones of $S(A_{n-1})$ containing $r$ are the cones over the edges of the root polytope $P(A_{n-1})$ containing $\e_i-\e_k$; these are known 
\cite{seashorethesis, cellini} to be:
\begin{equation}\label{eq:2cones}
        \mathcal{A}_{i,jk} = \text{cone}(\e_i-\e_j, \e_i-\e_k),
        \qquad        
        \mathcal{A}_{ij,k} = \text{cone}(\e_i-\e_k, \e_j-\e_k),
        \qquad
        \text{ for } j \neq i, k.
\end{equation}
One readily verifies that the primitive vectors in these cones with respect to $r$ are $\overline{\e_i-\e_j}$ and $\overline{\e_j-\e_k}$ in $\Z^n / \Z(\e_i-\e_k)$, respectively. 
Then the balancing condition for $r$ says
\[
 \sum_{\sigma \in \Sigma^2 \, : \, \sigma \supset r} w(\sigma)u_{\sigma/r} 
    = \sum_{k \neq i, j} \overline{(\e_{i} - \e_j)} + \sum_{k \neq i, j} \overline{(\e_j - \e_k)}
    = (n-2)\overline{(\e_i - \e_k)}
    = 0,
\]
as desired. It follows that $S(A_{n-1})$ is indeed a tropical surface.
\end{proof}

%
%
%
%


We can now state and prove our main result. 

\begin{theorem}
The degree of the tropical root surface $S(A_{n-1})$ is $\frac{1}{2} n (n-1)(n-2)$.
\end{theorem}

\begin{proof}
We follow the approach outlined at the end of Section \ref{sec:tropgeom}, studying the intersection of $v+S(A_{n-1})$ with $\Sigma_{n,n-2}$, where $v$ is the \emph{super-increasing} translation vector
\[
v=(0,1,10,100,1000,\ldots).
\]

It can be verified that this vector is generic by adding a small vector $\epsilon$ to it, and verifying that the intersection of $v+S(A_{n-1})$ with $\Sigma_{n,n-2}$, described below, has the same combinatorial structure as the intersection of $(v+\epsilon)+S(A_{n-1})$ with $\Sigma_{n,n-2}$. 

\bigskip

\textbf{A)} First we find the intersection points of $v+S(A_{n-1})$ and $\Sigma_{n,n-2}$.

\medskip

For each cone $\sigma \in S(A_{n-1})$ we need to find the points $v+s$ for $s \in \sigma$ whose three smallest entries are equal, so they are also in $\Sigma_{n,n-2}$.
The maximal cones of $S(A_{n-1})$ are of the form $\mathcal{A}_{i,jk}$ and $\mathcal{A}_{ij,k}$ for $i \neq j \neq k$, as defined in \eqref{eq:2cones}. We consider these two types of cones separately. 


\smallskip

\textbf{A1)} Let us find the intersection points of $v+\mathcal{A}_{i,jk}$ and $\Sigma_{n,n-2}$ for $i \neq j \neq k$. 

\smallskip

Let $s = a(\e_i-\e_j) + b(\e_i-\e_k)= (a+b) \e_i - a \e_j - b \e_k \in \A_{i,jk}$ for $a, b \geq 0$. To make the three smallest entries of $v+s$ equal, we need to choose one entry $i$ of $v=(0, 1, 10, 100, \ldots)$ to add $(a+b)$ to, and two entries $j$ and $k$ to subtract $a$ and $b$ from, respectively. Let $m = \min_{1 \leq i \leq n} (v+s)_i$ be the smallest coordinate of $v+s$, which appears at least three times; consider the following cases:

\smallskip

\textit{Case 1.1:} $m<0$: 

To achieve this minimum we would have to subtract from at least 3 entries of $v$. Since we can only subtract from 2 entries, this case does not contribute any intersection points.

\smallskip

\textit{Case 1.2:} $m = 0$: 

To achieve $m=0$, we must leave entry $v_1=0$ unchanged, subtract from any two other entries $j, k > 1$ (necessarily subtracting $a = v_j=10^{j-2}$ and $b= v_k=10^{k-2}$), and add (necessarily $a+b$) to any of the remaining entries $i \neq 1,j,k$. Thus, there are $\binom{n-1}{2}  (n-3)$ possible intersection points in this case.

\smallskip

\textit{Case 1.3:} $0<m<1$: 

To achieve such a value of $m$, we would have to add $a+b=m$ to $v_1$ and subtract $a=10^{j-2}-m$ and $b=10^{k-2}-m$ to two other entries $j, k > 1$. This would imply that $10^{j-2}+10^{k-2}=3m$, which is impossible because the left hand side is at least 11 and the right hand side is less than $3$. Thus this case does not contribute any intersection points.

\smallskip

\textit{Case 1.4:} $m=1$: 

In this case we must add at least 1 to entry $v_1=0$, leave entry $v_2=1$ unchanged, and subtract from two other entries $j,k > 2$. At least one of those two entries, say $j$, must lead to a minimum coordinate $(v+s)_j=1$, so $a = 10^{j-2}-1$. This means that $(v+s)_1 = a+b>1$ is not a minimum coordinate, so $(v+s)_k=1$ must be the other minimum coordinate, and $b= 10^{k-2}-1$. Thus, there are $\binom{n-2}{2}$ possible intersection points in this case.

\smallskip

\textit{Case 1.5:} $m>1$:

In this case we would have to add to the entries $v_1=0$ and $v_2=1$ to make them greater than or equal to $m$. Since we can only add to one entry, this case does not contribute any intersection points.

\medskip

\textbf{A2)} Now let us find the intersection points of $v+\mathcal{A}_{ij,k}$ and $\Sigma_{n,n-2}$ for $i \neq j \neq k$. 

\smallskip

Let $s = a(\e_i-\e_k) + b(\e_j-\e_k) = a \e_i + b \e_j - (a+b) \e_k \in \A_{ij,k}$ for $a, b \geq 0$. To make the three smallest entries of $v+s$ equal to each other, we 
need to choose two entries $i$ and $j$ of $v=(0, 1, 10, 100, \ldots)$ to add $a$ and $b$ to, respectively, and one entry $k$ to subtract $(a+b)$ from.
 Let $m = \min_{1 \leq i \leq n} (v+s)_i$ be the smallest coordinate of $v+s$, which appears at least three times. Consider the following cases:

\smallskip

\textit{Case 2.1:} $m<1$

To achieve this value of $m$ we would need to subtract from two of the original entries of $v$, which is impossible. Thus, this case does not contribute any intersection points.

\smallskip

\textit{Case 2.2:} $m=1$

A value of $m=1$ can only be achieved in entries $1,2,k$ of $v+s$ for some $k>2$.
We must add $a=1$ to $v_1$, leave $v_2$ unchanged, subtract $a+b=10^{k-2}-1$ from $v_k$, and hence add $b=10^{k-2}-2$ to some other entry $j \neq 1,2,k$. 
Thus, this case contributes $(n-2)(n-3)$ intersection points.

\smallskip

\textit{Case 2.3:} $1<m<10$

Again, such a value of $m$ can only be achieved in entries $1,2,k$ of $v+s$ for some $k>2$. Now, for these three new entries to equal $m$, we must add $a=m$ to $v_1$, add $b=m-1$ to $v_2$, and subtract $a+b=10^{k-2}-m$ from $v_k$. This forces $m+(m-1) = 10^{k-2}-m$, which gives $m= \frac13(10^{k-2}+1)$. Since $1<m<10$, we must have $k=3$. 
Thus, this case contributes a single intersection point.

\smallskip

\textit{Case 2.4:} $m=10$

In order to make the three smallest entries of $v+s$ equal to $10$,
we have the following three options. 1. Add $a=10$ to $v_1$, add $b=9$ to $v_2$, leave $v_3$ untouched, and subtract $19$ from any of the remaining entries $k >3$. 2. Add $a=10$ to $v_1$, subtract $a+b=10^{k-1} - 10$ from $v_k$ with $k>3$, and add $b=10^{k-1} - 20$ to $v_2$.  3. Add $a=9$ to $v_2$, subtract $a+b=10^{k-1} - 10$ from $v_k$ with $k>3$, and add $b=10^{k-1} - 19$ to $v_1$.
 In each of these options $k$ can be any number between $4$ and $n$, so this case contributes $3(n-3)$ intersection points.


\smallskip

\textit{Case 2.5:} $m>10$

To achieve this value of $m$, we would need to add to the three smallest entries of $v$, which is impossible. Thus, this case does not contribute any intersection points.

\medskip

\textbf{B)}
We now find the multiplicity of each of the intersection points $p$ that we found in A). To do that we first need to identify the cones of $v+S(A_{n-1})$ and $\Sigma_{n,n-2}$ that $p=v+s$ belongs to. Every cone in these fans is simplicial and saturated, so we can then compute the  multiplicity as the absolute value of the determinant of the matrix whose columns are the lattice generators of these cones. 

In the particular situation that interests us, there is a convenient general shortcut. Suppose that $p$ is the intersection point of cone $v+\A_{i,jk}$ (or $v+\A_{ij,k}$) of $v+S(A_{n-1})$ and cone $\sigma_{i_1,i_2,\ldots,i_{n-3}}$
of $\Sigma_{n,n-2}$ where $[n] - \{i_1, \ldots, i_{n-3}\} = \{i_{n-2},i_{n-1},i_{n}\}$.
Then the index of that intersection is 
\begin{eqnarray*}
\text{index}(p) &=& 
|\det(\e_{\{i_1\}}, \e_{\{i_1,i_2\}}, \e_{\{i_1,i_2,i_3\}} , \ldots, \e_{\{i_1,\ldots,i_{n-3}\}}, \e_S, \e_i-\e_j, \e_i-\e_k)| \\ 
&=& |\det(\e_{i_1}, \e_{i_2}, \e_{i_3} \ldots, \e_{i_{n-3}}, \e_{\{i_{n-2},i_{n-1}, i_n\}}, \e_i-\e_j, \e_i-\e_k)|.
\end{eqnarray*}
We reach this equality by performing elementary column operations that do not affect the determinant: we sequentially subtract the $(j-1)$th column from the $j$th column for $j=n-2, n-3, \ldots, 2$, thus replacing 
$\e_S$ with $\e_S - \e_{\{i_1,\ldots,i_{n-3}\}} = \e_{\{i_{n-2},i_{n-1}, i_n\}}$ and 
$\e_{\{i_1,\ldots,i_j\}}$  with $\e_{\{i_1,\ldots,i_j\}}  - \e_{\{i_1,\ldots,i_{j-1}\}} = \e_{i_j}$ for $n-3 \geq j \geq 2$. Notice that the resulting determinant only depends on the locations $i, j, k$ where we changed the entries of $v$, and the locations $i_{n-2}, i_{n-1}, i_n$ of the three smallest entries of $v+s$.
The analogous result holds for the cones of the form $v+\A_{ij,k}$.


\smallskip

\textbf{B1)}

\smallskip

\textit{Case 1.2:} In this case, we added to entry $i$ and subtracted from entries $j,k$ of $v$, and the three smallest (and equal) entries of $p=v+s$ are in positions $1,j,k$. The cones that intersect are $v+\A_{i,jk}$ and 
$\sigma_{i_1, \ldots, i_{n-3}}$
for a permutation $i_1, \ldots, i_{n-3}$ of $[n]-\{1,j,k\}$ that depends on the order of $i,j,k$.
Regardless of that permutation, the previous discussion tells us that 
\begin{eqnarray*}
\text{index}(p) 
&=& |\det(\widehat{\e_1}, \e_2, \e_3, \ldots, \e_i, \ldots, \widehat{\e_j}, \ldots, \widehat{\e_k}, \ldots, \e_n, \e_1+\e_j+\e_k, \e_i-\e_j, \e_i-\e_k)| \\ 
&=& |\det(\widehat{\e_1}, \e_2, \e_3, \ldots, \e_i, \ldots, \widehat{\e_j}, \ldots, \widehat{\e_k}, \ldots, \e_n, \e_1+\e_j+\e_k, \e_j, \e_k)| \\ 
&=& |\det(\widehat{\e_1}, \e_2, \e_3, \ldots, \e_i, \ldots, \widehat{\e_j}, \ldots, \widehat{\e_k}, \ldots, \e_n, \e_1, \e_j, \e_k)| \\ 
&=& 1,
\end{eqnarray*}
where in the first step we subtract column $\e_i$ from columns $\e_i-\e_j$ and $\e_i-\e_k$ and change their signs, and in the second step we subtract $\e_j$ and $\e_k$ from $\e_1+\e_j+\e_k$. Thus all intersections in this case have multiplicity $1$.

\smallskip

\textit{Case 1.4:} In this case, we added to entry $1$ and subtracted from entries $j,k$ of $v$, and the smallest entries of $p=v+s$ are in positions $2,j,k$. Similarly to Case 1.2, we obtain
\begin{eqnarray*}
\text{index}(p) 
&=& |\det(\e_1, \widehat{\e_2}, \e_3, \ldots, \widehat{\e_j}, \ldots, \widehat{\e_k}, \ldots, \e_n, \e_2+\e_j+\e_k, \e_1-\e_j, \e_1-\e_k)| \\ 
&=& |\det(\e_1, \widehat{\e_2}, \e_3, \ldots, \widehat{\e_j}, \ldots, \widehat{\e_k}, \ldots, \e_n, \e_2+\e_j+\e_k, \e_j, \e_k)| \\ 
&=& |\det(\e_1, \widehat{\e_2}, \e_3, \ldots, \widehat{\e_j}, \ldots, \widehat{\e_k}, \ldots, \e_n, \e_2, \e_j, \e_k)| \\ 
&=& 1.
\end{eqnarray*}
It follows that each intersection in Case 1.4 has multiplicity $1$.

\smallskip

\textbf{B2)}

\smallskip

\textit{Case 2.2}: 
In this case we added to entries $i=1, j$ and subtracted from entry $k$ of $v$, and we ended up with the three smallest entries of $p$ in positions $1,2,k$. Therefore
\begin{eqnarray*}
\text{index}(p) 
&=& |\det(\widehat{\e_1}, \widehat{\e_2}, \e_3, \e_4, \ldots,  \widehat{\e_k}, \ldots, \e_n, \e_1+\e_2+\e_k, \e_1-\e_k, \e_j-\e_k)| \\ 
&=& |\det(\widehat{\e_1}, \widehat{\e_2}, \e_3, \e_4, \ldots,  \widehat{\e_k}, \ldots, \e_n, \e_1+\e_2+\e_k, \e_1-\e_k, \e_k)| \\ 
&=& |\det(\widehat{\e_1}, \widehat{\e_2}, \e_3, \e_4, \ldots,  \widehat{\e_k}, \ldots, \e_n, \e_1+\e_2+\e_k, \e_1, \e_k)| \\ 
&=& |\det(\widehat{\e_1}, \widehat{\e_2}, \e_3, \e_4, \ldots,  \widehat{\e_k}, \ldots, \e_n, \e_2, \e_1, \e_k)| \\ 
&=& 1,
\end{eqnarray*}
where we first subtract $\e_j$ from $\e_j-\e_k$ and change the sign to $\e_k$, then we add $\e_k$ to $\e_1-\e_k$, and then we subtract $\e_1$ and $\e_k$ from $\e_1+\e_2+\e_k$.
Again, it follows that each one of these intersections has multiplicity $1$.

\smallskip

\textit{Case 2.3}:
Here we added to entries $i=1, 2$ and subtracted from entry $3$ of $v$, and we ended up with the three smallest entries of $p$ in positions $1,2,3$. Therefore
\begin{eqnarray*}
\text{index}(p) 
&=& |\det(\widehat{\e_1}, \widehat{\e_2}, \widehat{\e_3}, \e_4, \ldots,  \e_n, \e_1+\e_2+\e_3, \e_1-\e_3, \e_2-\e_3)| \\ 
&=& |\det(\widehat{\e_1}, \widehat{\e_2}, \widehat{\e_3}, \e_4, \ldots,  \e_n, 3\e_3, \e_1-\e_3, \e_2-\e_3)| \\ 
&=& |\det(\widehat{\e_1}, \widehat{\e_2}, \widehat{\e_3}, \e_4, \ldots,  \e_n, 3\e_3, \e_1, \e_2)| \\ 
&=& 3,
\end{eqnarray*}
where we first subtract $\e_1-\e_3$ and $\e_2-\e_3$ from $\e_1+\e_2+\e_3$, and then we add one third of $3\e_3$ to $\e_1-\e_3$ and $\e_2-\e_3$. Thus this intersection has multiplicity $3$.

\smallskip

\textit{Case 2.4}: Here we had three options:
In option 1 we added to entries $i=1, 2$ and subtracted from entry $k$ of $v$, and we ended up with the three smallest entries of $p$ in positions $1,2,3$. Therefore
\begin{eqnarray*}
\text{index}(p) 
&=& |\det(\widehat{\e_1}, \widehat{\e_2}, \widehat{\e_3}, \e_4, \ldots, \e_k, \ldots,  \e_n, \e_1+\e_2+\e_3, \e_1-\e_k, \e_2-\e_k)| \\ 
&=& |\det(\widehat{\e_1}, \widehat{\e_2}, \widehat{\e_3}, \e_4, \ldots, \e_k, \ldots, \e_n, \e_1+\e_2+\e_3, \e_1, \e_2)| \\ 
&=& |\det(\widehat{\e_1}, \widehat{\e_2}, \widehat{\e_3}, \e_4, \ldots, \e_k, \ldots, \e_n, \e_3, \e_1, \e_2)| \\ 
&=& 1,
\end{eqnarray*}
where we first add $\e_k$ to $\e_1-\e_k$ and $\e_2-\e_k$, and then subtract $\e_1$ and $\e_2$ from $\e_1+\e_2+\e_3$. These intersections then have multiplicity $1$.

In option 2 we added to entries $i=1, 2$ and subtracted from entry $k$ of $v$, and we ended up with the three smallest entries of $p$ in positions $1,3,k$. Therefore
\begin{eqnarray*}
\text{index}(p) 
&=& |\det(\widehat{\e_1}, \e_2, \widehat{\e_3}, \e_4, \ldots,  \widehat{\e_k}, \ldots, \e_n, \e_1+\e_3+\e_k, \e_1-\e_k, \e_2-\e_k)| \\ 
&=& |\det(\widehat{\e_1}, \e_2, \widehat{\e_3}, \e_4, \ldots,  \widehat{\e_k}, \ldots, \e_n, \e_1+\e_3+\e_k, \e_1-\e_k, \e_k)| \\ 
&=& |\det(\widehat{\e_1}, \e_2, \widehat{\e_3}, \e_4, \ldots,  \widehat{\e_k}, \ldots, \e_n, \e_1+\e_3+\e_k, \e_1, \e_k)| \\ 
&=& |\det(\widehat{\e_1}, \e_2, \widehat{\e_3}, \e_4, \ldots,  \widehat{\e_k}, \ldots, \e_n, \e_3, \e_1, \e_k)| \\
&=& 1,
\end{eqnarray*}
where we first subtract $\e_2$ from $\e_2-\e_k$ and change the sign of the result, then add $\e_k$ to $\e_1-\e_k$, and finally subtract $\e_1$ and $\e_k$ from $\e_1+\e_3+\e_k$. These intersections then have multiplicity $1$.

Option 3 is analogous to option 2, reversing the roles of $1$ and $2$, so these intersections have multiplicity $1$ as well.


\medskip

\textbf{C)} Finally, we collect  in Table \ref{tab:my_label} all the intersections points and their multiplicities, as computed in A) and B).
Putting them together, we conclude that the degree of the tropical root surface of type $A_{n-1}$ is
\begin{eqnarray*}
\deg S(A_{n-1}) &=&  \binom{n-1}{2}  (n-3) + \binom{n-2}{2} + (n-2)(n-3) + 3 + 3(n-3) \\
&=& \frac{1}{2} n (n-1)(n-2),
\end{eqnarray*}
as desired.
\end{proof}

\begin{table}[ht] 
    \centering
    \begin{tabular}{|c|c|c|c|}
         \hline
         case & number of intersections & intersection multiplicity & contribution to degree \\
         \hline
         1.1 & 0 & - & 0\\
        \hline
         1.2 & $\binom{n-1}{2}  (n-3)$ & 1 & $\frac{1}{2}(n-1)(n-2)(n-3)$\\
         \hline
         1.3 & 0 & - & 0\\
         \hline
         1.4 & $\binom{n-2}{2}$ & 1 & $\frac{1}{2}(n-2)(n-3)$\\
         \hline
         1.5 & 0 & - & 0\\
         \hline
         2.1 & 0 & - & 0\\
         \hline
         2.2 & $(n-2)(n-3)$ & 1 & $(n-2)(n-3)$\\
         \hline
         2.3 & 1 & 3 & 3\\
         \hline
         2.4 & $3(n-3)$ & 1 & $3(n-3)$\\
         \hline
         2.5 & 0 & - & 0\\
         \hline
                  
    \end{tabular}
    \caption{Intersection points of $v+S(A_{n-1})$ and $\Sigma_{n,n-2}$ with their multiplicities.}
    \label{tab:my_label}
\end{table}

\section{\textsf{Future Work}} 

In future work we plan to compute the degrees of the tropical surfaces of the remaining root systems. We also plan to determine whether these tropical surfaces can be obtained as tropicalizations of algebraic varieties.

\section{\textsf{Acknowledgements}}

We would like to thank Chiemi Kato, Jewell McMillon, Maria Isabel Perez, and Anna Schindler for many valuable discussions on tropical root surfaces. The project \cite{poster} and this one grew out of a conversation with June Huh about \cite{babaeehuh} and the desire to construct interesting examples of tropical surfaces (and more generally tropical varieties) that do not obviously arise via tropicalization. We also thank Nayeong Kim and Johannes Rau for useful discussions on tropical intersection theory. 

\small

\bibliography{biblio}

\bibliographystyle{plain}



\end{document}